\documentclass[12pt]{amsart}
\usepackage{amssymb}
\usepackage{enumitem}
\usepackage{graphicx}

\makeatletter
\@namedef{subjclassname@2020}{%
  \textup{2020} Mathematics Subject Classification}
\makeatother

\usepackage[T1]{fontenc}
\newtheorem{theorem}{Theorem}[section]

\newtheorem{lemma}[theorem]{Lemma}

\theoremstyle{definition}

\newtheorem{remark}[theorem]{Remark}

\numberwithin{equation}{section}
\def\pmod #1{\ ({\rm{mod}}\ #1)}
\def\Z{\Bbb Z}
\def\N{\Bbb N}

\def\R{\Bbb R}

\def\l{\left}
\def\r{\right}
\def\bg{\bigg}
\def\({\bg(}
\def\){\bg)}
\def\t{\text}
\def\f{\frac}

\def\gen{{\rm gen}}
\def\spn{{\rm spn}}
\def\Aut{{\rm Aut}}

\def\ls{\leqslant}
\def\gs{\geqslant}
\def\se {\subseteq}

\def\ve{\varepsilon}

\def\eq{\equiv}

\def\Proof{\noindent{\it Proof}}

\theoremstyle{plain}

\begin{document}

\baselineskip=17pt
\hbox{Acta Arith.  193 (2020), no.\,3, 253--268.}
\medskip

\title[On the 1-3-5 conjecture and related topics]{On the 1-3-5 conjecture and related topics}

\author[H.-L. Wu and Z.-W. Sun]{Hai-Liang Wu and Zhi-Wei Sun}

\address {(Hai-Liang Wu)  Department of Mathematics, Nanjing
University, Nanjing 210093, People's Republic of China}
\email{whl.math@smail.nju.edu.cn}

\address
{(Zhi-Wei Sun, corresponding author) Department of Mathematics
\\ Nanjing University\\
 Nanjing 210093\\ People's Republic of China}
\email {zwsun@nju.edu.cn}

\date{}

\begin{abstract}
The 1-3-5 conjecture of Z.-W. Sun states that any $n\in\N=\{0,1,2,\ldots\}$ can be
written as $x^2+y^2+z^2+w^2$ with $w,x,y,z\in\N$ such that $x+3y+5z$ is a square. In this paper, via the theory of ternary quadratic forms and related modular forms,
we study the integer version of the 1-3-5 conjecture and related weighted sums of four squares with
certain linear restrictions.
Here are two typical results in this paper:
(i) There is a finite set $A$ of positive integers such that any sufficiently large integer not in the
set $\{16^ka:\ a\in A,\ k\in\N\}$ can be written as
$x^2+y^2+z^2+w^2$ with $x,y,z,w\in\Z$ and $x+3y+5z\in\{4^k:\ k\in\N\}$.
(ii) Any positive integer can be written as $x^2+y^2+z^2+2w^2$ with $x,y,z,w\in\Z$ and $x+y+2z+2w=1$.
Also, any sufficiently large integer can be written as $x^2+y^2+z^2+2w^2$ with $x,y,z,w\in\Z$ and
$x+2y+3z=1$.
\end{abstract}



\maketitle

\section{Introduction}

In 1770 Lagrange established his celebrated four-square theorem which states that any
$n\in\N=\{0,1,2,\ldots\}$ can be written as the sum of four squares.
In 1917 S. Ramanujan \cite{R} listed 55 possible quadratuples $(a,b,c,d)$ of positive integers with $a\ls
b\ls c\ls d$ such that
$$\{ax^2+by^2+cz^2+dw^2:\ x,y,z,w\in\Z\}=\N.$$
In 1927 L. E. Dickson \cite{D27} confirmed 54 of them and pointed out that the remaining one on
Ramanujan's list is wrong.

Recently, Z.-W. Sun \cite{S17a, S17b} refined Lagrange's four-square theorem in various ways and posed
many conjectures on restricted sums of four squares. In particular, he \cite{S17a} would like to offer
\$1350 (USD) for the first solution of his following conjecture.
\medskip

\noindent{\bf 1-3-5 Conjecture} (Sun \cite{S17a}). Any $n\in\N$ can be written as $x^2+y^2+z^2+w^2$
($x,y,z,w\in\N$) with $x+3y+5z$ a square.
\medskip

This conjecture looks quite challenging, and it has been verified for $n$ up to $10^{10}$ by Q.-H. Hou at
Tianjin Univ. In this direction, Y.-C. Sun and Z.-W. Sun \cite[Theorem 1.8]{SS} used Euler's four-square
identity to show that any $n\in\N$ can be written as
$x^2+y^2+z^2+w^2$ with $x,y,5z,5w\in\Z$ such that $x+3y+5z$ is a square.

In \cite{S17b} Z.-W. Sun investigated weighted sums of four squares with certain linear restrictions, for
example, he showed that any $n\in\Z^+=\{1,2,3,\ldots\}$ can be written as $x^2+y^2+z^2+2w^2$ with
$x,y,z,w\in\Z$ and $y+z+w=1$, and conjectured that each $n\in\Z^+$
can be written as $x^2+y^2+z^2+2w^2$ with $x,y,z,w\in\Z$ and $x+2y+3z=1$.

In this paper we study the integer version of the 1-3-5 conjecture and some conjectures of Sun
\cite[Conjectures 4.12-4.13]{S17b} along this line. We establish the following two theorems via the theory
of ternary quadratic forms and related modular forms.

\begin{theorem} \label{Th1.1}
{\rm (i)} Any $n\in \Z^+$ can be written as $x^2+y^2+z^2+2w^2$ $(x,y,z,w\in \Z)$ with any of the following
restrictions:
\begin{equation*}
y+3z+2w=1,\ x+y+2z+2w=1,\ x+y+2z+2w=2.
\end{equation*}

{\rm (ii)} Let $\lambda\in\{1,3\}$. Then any $n\in \Z^+$ can be written as $x^2+y^2+z^2+3w^2$ $(x,y,z,w\in
\Z)$ with $2y+z+\lambda w=1$.

{\rm (iii)} Any $n\in \Z^+$ can be written as $x^2+y^2+2z^2+3w^2$ $(x,y,z,w\in \Z)$ with $y+2z+3w=1$.
Also, each $n\in \Z^+$ can be written as $x^2+y^2+3z^2+4w^2$ $(x,y,z,w\in\Z)$ with $y+z+2w=1$.
\end{theorem}
\begin{remark}\label{Rem1.1}
Our proof of Theorem \ref{Th1.1} uses some regular ternary quadratic forms as well as the genus theory of
ternary quadratic forms.\end{remark}

\begin{theorem} \label{Th1.2}
{\rm (i)} Any sufficiently large integer $n$ with $16\nmid n$ can be written as $x^2+y^2+z^2+w^2$ with
$x,y,z,w\in \Z$ and $x+3y+5z\in \{1,4\}$.

{\rm (ii)} Any sufficiently large integer $n$ can be written as $x^2+y^2+z^2+2w^2$ with $x,y,z,w\in \Z$
and $x+\mu y+\nu z=1$, where $(\mu,\nu)$ is any of the ordered pairs
$(2,3),\ (2,5)$ and $(3,4)$. Also, each sufficiently large integer $n$ can be written as
$x^2+y^2+z^2+2w^2$ with $x,y,z,w\in \Z$ and $2y+z+w=1$.

{\rm (iii)} Each sufficiently large integer $n$ can be written as $x^2+y^2+2z^2+3w^2$ with $x,y,z,w\in \Z$
and $P(x,y,z,w)=1$, whenever $P(x,y,z,w)$ is among the polynomials
\begin{align*}
x+2y+w,\ y+z+w,\ y+2z+w.
\end{align*}
\end{theorem}
\begin{remark}\label{Rem1.2}
Our proof of Theorem \ref{Th1.2} involves the theory of spinor exceptional square classes for quadratic
forms and some significant results on modular forms of weight $3/2$.
\end{remark}

Clearly, Theorem \ref{Th1.2}(i) implies that there is a finite set $A\se\Z^+$
such that any sufficiently large integer not in the set $\{16^ka:\ a\in A,\, k\in\N\}$ can be written as
$x^2+y^2+z^2+w^2$
with $x,y,z,w\in\Z$ and $x+3y+5z\in\{4^k:\,k\in\N\}$. This provides remarkable progress on the integer
version of the 1-3-5 conjecture.
Note that Sun \cite[Conjecture 4.5(ii)]{S17b} conjectured that any $n\in\Z^+$ can be written as
$x^2+y^2+z^2+w^2$
with $x,y,z,w\in\N$ and $|x+3y-5z|\in\{4^k:\,k\in\N\}$.
Sun \cite [Conjectures 4.12-4.13]{S17b} also conjectured that Theorem \ref{Th1.2} remains valid
if we don't require $n\in\Z^+$ sufficiently large, but we have been unable to prove this.

We will give in Section 2 a brief overview of the theory of ternary quadratic forms and modular
forms of weight $3/2$ which we need in later proofs, and show Theorem \ref{Th1.1} and Theorem \ref{Th1.2} in Sections 3 and 4 respectively.

\section{Some preparations}

Let
\begin{equation}\label{2.1} f(x,y,z)=ax^2+by^2+cz^2+ryz+szx+txy
\end{equation}
be a positive definite ternary quadratic form with integral coefficients. Its associated matrix  is
 $$A=\begin{pmatrix} 2a & t &s \\t & 2b &r \\ s & r &2c \end{pmatrix}.$$
The {\it discriminant} of $f$ is given by $d(f):=\det(A)/2$.

Our proof of Theorem \ref{Th1.1} is closely related to the arithmetic theory of quadratic forms. The following
lemma is one of the most important results on integral representations of quadratic forms (cf. \cite [pp.
129]{C}).
\begin{lemma}\label{Lem2.1}
Let $f$ be a nonsingular integral quadratic form and let $m$ be a nonzero integer represented by $f$ over
the real field $\R$
and the ring $\Z_p$ of $p$-adic integers for each prime $p$.
Then $m$ is represented by some form $f^*$ over $\Z$ with $f^*$ in the same genus of $f$.
\end{lemma}
\begin{remark}\label{Rem2.1}
According to the effective methods in \cite [pp. 186--187]{Jones}, one may easily obtain all the eligible
numbers of the genus of an integral ternary quadratic form $f$.
\end{remark}

A positive definite ternary quadratic form $f$ is called {\it regular} if it represents every integer represented
by the genus of $f$.
Dickson was the first to study regular quadratic forms systematically and he \cite [pp. 112--113]{D39} listed all
the 102  diagonal regular ternary quadratic forms $ax^2+by^2+cz^2$
with $1\ls a\ls b\ls c$ and $\gcd(a,b,c)=1$, together with the structure of
$$E(a,b,c)=\{n \in\N:\ n\not= ax^2+by^2+cz^2\ \mbox{for all}\ x,y,z\in\Z\}.$$

In the present paper, we need the following well known results from \cite [pp. 112--113]{D39}.
\begin{lemma} {\rm (Dickson \cite[pp. 112-113]{D39})} \label{Lem2.2} We have

\begin{align*} E(1,1,2)=&\{4^k(16m+14): k,m\in \N\},
\\ E(1,1,6)=&\{9^k(9m+3) : k,m\in \N\},
\\E(1,2,6)=&\{4^k(8m+5) : k,m\in \N\},
\\ E(2,3,3)=&\{9^k(3m+1) : k,m\in \N\},
\\E(1,1,16)=&\bigcup_{k,m\in\N}\{4^k(8m+7),4m+3,8m+6,32m+12\}.
\end{align*}
\end{lemma}

To understand our proof of Theorem \ref{Th1.2}, we need to introduce the theory of spinor exceptional square
classes for quadratic forms and some significant results of modular forms of weight $3/2$. (The readers
may consult \cite{DR, Hanke, KS, SP80, Siegel}).

 Suppose that $a\in \Z$ is represented by the genus of $f$ given by (\ref{2.1}).
  M. Kneser \cite{Ke} proved that $a$ is represented either by all or by exactly half of the spinor genera in
  the genus. We call $a$  {\it spinor exceptional} for the genus of $f$ if $a$ is represented by exactly
  half of these spinor genera. R. Schulze-Pillot \cite{SP80} determined completely when $a$ is spinor
  exceptional for the genus of $f$. A. G. Earnest, J. S. Hisa and D. C. Hung \cite{EHH} found a similar
  characterization of primitively spinor exceptional numbers.
In view of the work in \cite{DR, EHH, SP00}, there are only finitely many spinor exceptional square
classes for each positive definite ternary quadratic form.

Now, we introduce some basic notations which can be found in \cite{C, Ki, Oto}.
For the positive definite ternary quadratic form $f$ given by (\ref{2.1}), $\Aut(f)$ denotes the group of integral
isometries of $f$.
 For $n\in\N$, write
 $$r(f,n):=\#\{(x,y,z)\in \Z^3 : f(x,y,z)=n\}$$
 (where $\#S$ denotes the cardinality of a set $S$), and let
  \begin{equation*}r(\gen(f),n):=\bigg(\sum_{f^*\in \gen(f)}\f1{\#\Aut(f^*)}\bigg)^{-1}\sum_{f^*\in \gen(f)}\frac{r(f^*,n)}{\# \Aut(f^*)},
  \end{equation*}
  where the summation is over a set of representatives of the classes in the genus of $f$.
  Similarly, we define
  \begin{equation*}
  r(\spn(f),n):=\bigg(\sum_{f^*\in \spn(f)}\f1{\#\Aut(f^*)}\bigg)^{-1}\sum_{f^*\in \spn(f)}\frac{r(f^*,n)}{\# \Aut(f^*)},
  \end{equation*}
  where the summation is over a set of representatives of the classes in the spinor genus of $f$.
  By \cite{DR}, for each $n\in\Z^+$ not in the spinor exceptional square classes we have
  \begin{equation*}
  r(\gen(f),n)=r(\spn(f),n).
  \end{equation*}
It is well known that the theta series
\begin{equation*}
\theta_f(z)=\sum_{n\gs0}r(f,n)e^{2\pi inz}
\end{equation*}
(with $z$ in the upper half plane) is a modular form of weight $3/2$ (cf. \cite{Shimura}). According to W. Duke and R. Schulze-Pillot's celebrated work \cite{DR}, if we write
\begin{equation*}
\theta_f=\theta_{\gen(f)}+(\theta_{\spn(f)}-\theta_{\gen(f)})+(\theta_f-\theta_{\spn(f)}),
\end{equation*}
 then $\theta_{\gen(f)}$ is an Eisenstein series, $\theta_{\spn(f)}-\theta_{\gen(f)}$ is a cusp form with
 its Fourier coefficients
 supported at finitely many spinor exceptional square classes, and on the other hand, $\theta_f-\theta_{\spn(f)}$ is a
 cusp form whose Shimura's lift is also a cusp form.

 We now give a brief discussion on the Fourier coefficients of the above Eisenstein series.
 According to the Siegel-Minkowski formula, we have
 \begin{equation*}
 r(\gen(f),n)=2\pi\sqrt{\frac{4n}{d(f)}}\ \prod_p\alpha_p(f,n),
 \end{equation*}
where $p$ runs over all primes and $\alpha_p(f,n)$ is the local density
(for more details on local density, the readers may consult \cite[Section 5.6]{Ki}) given by
\begin{equation*}
\alpha_p(f,n):=\lim_{k\to +\infty}p^{-2k}\#\{(x,y,z)\in(\Z/p^k\Z)^3: \ f(x,y,z)=n+p^k\Z\}.
\end{equation*}

If $n\in\N$ is not represented by $f$ locally, then the $n$-th coefficient of $\theta_{\gen(f)}$ vanishes.
If $n\in\N$ is represented by the genus of $f$ and $n$ has bounded divisibility at each anisotropic prime,
then by \cite[Lemma 5]{DR} we have
$r(\gen(f),n) \gg n^{1/2-\ve}$ for any $\ve>0$;
the bound here is ineffective because it relies on Siegel's lower bound for the class number of an arbitrary imaginary quadratic field.

We now turn to the cusp form. According to \cite{Duke}, if $n\in\Z^+$ is not in any of the spinor
exceptional square classes, then for all $\ve >0$,
the $n$-th Fourier coefficient of the cusp form $\theta_f-\theta_{\spn(f)}$ are  $\ll
n^{1/2-1/28+\ve}$. (This bound is effective as pointed out by the referee.)

In view of the above, we have the following well-known lemma (cf. \cite{DR}).
\begin{lemma}\label{Lem2.3} Let $f$ be a positive definite ternary quadratic form given by $(\ref{2.1})$. If $n$ is
represented by the genus of $f$ and not in any
 of the spinor exceptional square classes, and $n$ has bounded divisibility at each anisotropic prime,
 then $n$ is represented by $f$ over $\Z$ provided that $n$ is large enough.
\end{lemma}

We also need the following basic result on the number of spinor genera in a genus.

\begin{lemma}\label{Lem2.4} {\rm (\cite[p.\,202]{C})} Let $f$ be a positive definite ternary quadratic form given
by $(\ref{2.1})$.
Suppose that the genus of $f$ contains at least two spinor genera. Then either $r,s,t$ are all even and
$d(f)$ is a multiple of $16$,
or $p^3\mid d(f)$ for some odd prime $p$.
\end{lemma}

\section{Proof of Theorem \ref{Th1.1}}

\begin{lemma}\label{Lem3.1}
For any $n\in \Z^+$, we can write $88n-11$ as $x^2+8y^2+44z^2$ with $x,y,z\in\Z$ and $x\equiv 1\pmod 8$.
\end{lemma}
\begin{proof}
By Lemma \ref{Lem2.2}, there are $a,b,c\in\Z$ such that $8n-1=a^2+2b^2+(2c)^2$.
 As $a^2+2b^2\eq-1\pmod4$, we have $2\nmid ab$. Without loss of generality, we may assume that $a\eq1\pmod
 4$ (otherwise we may replace $a$ by $-a$)
 and that $b\eq (3a-1)/2\pmod 4$ (since $(3a-1)/2$ is odd). Therefore
\begin{equation*}
88n-11=11(8n-1)=(3a-2b)^2+8\l(\frac{a+3b}{2}\r)^2+44c^2
\end{equation*}
with $3a-2b\eq1\pmod8$. This concludes the proof.
\end{proof}

\begin{lemma}\label{Lem3.2} For any $n\in\Z^+$, we can write
$88n-44$ as  $x^2+8y^2+44z^2$ with $x,y,z\in\Z$ and $x\equiv 2\pmod 8$.
\end{lemma}
\begin{proof}
By Lemma \ref{Lem2.2}, we can write $2n-1=a^2+2b^2+4c^2$ with $a,b,c\in\Z$. Without loss of generality, we
assume that $a\eq(-1)^{b-1}\pmod4$.
Clearly,
\begin{equation*}
88n-44=44(2n-1)=(6a-4b)^2+8(a+3b)^2+44(2c)^2
\end{equation*}
and
$$6a-4b\eq 6(-1)^{b-1}-4b\eq 2(-1)^b+4b\eq2\pmod 8.$$
So the desired result follows.
\end{proof}
\medskip
\noindent{\it Proof of Theorem} \ref{Th1.1}(i). (1) We first consider the restriction $y+3z+2w=1$.
By  Lemma \ref{Lem2.2}, we can write $12n-1=2r^2+3s^2+3t^2$ with $r,s,t\in\Z$.
Clearly, $3\nmid r$ and $s\not \equiv t\pmod 2$.
Since $-1\equiv 2r^2+3\pmod 4$,  we have $2\mid r$. Without loss of generality, we may further assume that
$$ r=2(3u+1),\ s=2x,\ t=4z-1$$
with $u,x,z\in\Z$. Then
$$12n-1=2(2(3u+1))^2+3(2x)^2+3(4z-1)^2$$
and hence
$$n=x^2+(1-z+2u)^2+z^2+2(-z-u)^2.$$
Let $y=1-z+2u$ and $w=-z-u$. Then $y+3z+2w=1$ as desired.

(2) Now we consider the restriction $x+y+2z+2w=1$.
By Lemma \ref{Lem3.1},  $88n-11=r_1^2+44r_2^2+8(s+11t)^2$ for some $r_1,r_2,s,t\in\Z$ with $r_1\equiv
1\pmod 8$, $s\eq 2r_1\pmod{11}$ and $0\ls s\ls10$.
(Note that $-r_1^2\eq 8(2r_1)^2\pmod{11}$.)
As $r_1^2+44r_2^2\eq-11\eq5\pmod8$, we have $2\nmid r_2$.
Write $r_2=2u-1$ if $s\ls 6$, and $r_2=2u-3$ if $s\in\{7,8,9,10\}$.
Also, we may write $r_1=88v+s^*$ with $v\in\Z$, where
$$s^*=\begin{cases}33-16s&\mbox{if}\ s\in\{0,\ldots,6\},\\121-16s&\mbox{if}\
s\in\{7,8,9,10\}.\end{cases}$$
As $88n-11=8(s+11t)^2+(88v+s^*)^2+44r_2^2$, we get
$$n=(a+2t-u+v)^2+(2t+u+v)^2+(b-t-6v)^2+2(c-t+5v)^2,$$
where
$$(a,b,c)=\begin{cases}(1,s-2,2-s)&\mbox{if}\ s\in\{0,\ldots,6\},
\\(3,s-8,7-s)&\mbox{if}\ s\in\{7,8,9,10\}.
\end{cases}$$
Set
$$x=a+2t-u+v,\ y=2t+u+v,\ z=b-t-6v,\ w=c-t+5v.$$
Then $n=x^2+y^2+z^2+2w^2$ and
$$x+y+2z+2w=a+2b+2c=1.$$

(3) Finally we consider the restriction $x+y+2z+2w=2$.
By Lemma \ref{Lem3.2}, we can write $88n-44$ in the form $r_1^2+8(s+11t)^2+44r_2^2$ with $r_1\equiv 2\pmod
8$, $s\eq 2r_1\pmod{11}$ and $0\ls s\ls 10$.
(Note that $-r_1^2\eq 8(2r_1)^2\pmod{11}$.) As $r_1^2\eq 2^2\eq-44\pmod8$, we see that $r_2=2u$ for some
$u\in\Z$.
Also, we may write $r_1=88v-16s-22$ with $v\in\Z$. Thus
$$88n-44=(88v-16s-22)^2+8(s+11t)^2+44(2u)^2$$
and hence
$$n=(2t-u+v)^2+(2t+u+v)^2+(2+s-t-6v)^2+2(-1-s-t+5v)^2.$$
Set
$$x=2t-u+v,\ y=2t+u+v,\ z=2+s-t-6v,\ w=-1-s-t+5v.$$
Then $n=x^2+y^2+z^2+2w^2$ and $x+y+2z+2w=2$.

In view of the above, we have completed the proof of Theorem \ref{Th1.1}(i). \qed

\begin{lemma}\label{Lem3.3} For any $n\in\Z^+$, we can write $80n-15$ as $x^2+16y^2+80z^2$ with
$x,y,z\in\Z$.
\end{lemma}
\begin{proof} By Lemma \ref{Lem2.2}, we can write $16n-3=a^2+b^2+16c^2$ with $a,b,c\in\Z$. Clearly, $a$
and $b$ cannot be both even.
Without loss of generality, we assume that $a$ is odd.
Since $-3\equiv a^2+b^2\pmod{8}$, we have $b\equiv 2\eq2a\pmod 4$. Therefore
\begin{equation*}
80n-15=5(16n-3)=(a+2b)^2+16\l(\frac{2a-b}{4}\r)^2+80c^2
\end{equation*}
with $(2a-b)/4$ integral. This concludes the proof.
\end{proof}
\medskip

\begin{lemma}\label{Lem3.4} For any $n\in\Z^+$, we can write $120n-15$ as $x^2+6y^2+30z^2$ with $x,y,z\in
\Z$.
\end{lemma}
\begin{proof} By Lemma \ref{Lem2.2}, there are $a,b,c\in\Z$ such that $8n-1=a^2+2b^2+6c^2$. Thus
\begin{equation*}
120n-15=15(8n-1)=(3a-6c)^2+6(a+3c)^2+30b^2.
\end{equation*}
This proves the desired result.
\end{proof}
\medskip
\noindent{\it Proof of Theorem} \ref{Th1.1}(ii). (1) We first handle the case $\lambda=1$.
By Lemma \ref{Lem3.3}, we can write $80n-15=r_1^2+16r_2^2+80x^2$ with $r_1,r_2,x\in\Z$.
Since $r_1^2\eq-15\eq1\pmod{16}$, we have $r_1\eq\pm1\pmod{8}$. Without loss of generality, we assume that
$r_1\eq1\pmod8$.
As $r_2^2\eq -r_1^2\eq(2r_1)^2\pmod{5}$, without loss of generality we may assume that $r_2=s+5t$ with
$s,t\in\Z$, $0\ls s\ls4$ and $s\eq-2r_1\pmod5$.

{\it Case} 1. $r_1\eq1\pmod{16}$.

As $5-(-1)^s20-8s\eq 2s\pmod5$ and $5-(-1)^s20-8s\eq 1\pmod{16}$, we can write $r_1=80v+5-(-1)^s20-8s$
with $v\in\Z$.
Thus $$80n-15=(80v+5-(-1)^s20-8s)^2+16(s+5t)^2+80x^2$$ and hence
$$n=x^2+(a+t+2v)^2+(b-2t+v)^2+3(c-5v)^2$$
where
$$(a,b,c)=\begin{cases}(0,-s/2,1+s/2)&\mbox{if}\ s\in\{0,2,4\},
\\(1,(1-s)/2,(s-3)/2)&\mbox{if}\ s\in\{1,3\}.
\end{cases}$$
Let
$$y=a+t+2v, \ z=b-2t+v,\ w=c-5v.$$
Then $n=x^2+y^2+z^2+3w^2$ and
$$2y+z+w=2(a+t+2v)+(b-2t+v)+(c-5v)=2a+b+c=1.$$

{\it Case} 2. $r_1\eq9\pmod{16}$.

As $5+(-1)^s20-8s\eq 2s\pmod5$ and $5+(-1)^s20-8s\eq 9\pmod{16}$, we can write $r_1=80v+5+(-1)^s20-8s$
with $v\in\Z$.
Thus
$$80n-15=(80v+5+(-1)^s20-8s)^2+16(s+5t)^2+80x^2$$ and hence
$$n=x^2+(a+t+2v)^2+(b-t-7v)^2+3(c-t+3v)^2$$
where
$$(a,b,c)=\begin{cases}(1,s/2-2,1-s/2)&\mbox{if}\ s\in\{0,2,4\},
\\(0,(s+3)/2,-(s+1)/2)&\mbox{if}\ s\in\{1,3\}.
\end{cases}$$
Let
$$y=a+t+2v,\ z=b-t-7v,\ w=c-t+3v.$$ Then $n=x^2+y^2+z^2+3w^2$ and
$$2y+z+w=2(a+t+2v)+(b-t-7v)+(c-t+3v)=2a+b+c=1.$$

(2) Now we handle the case $\lambda=3$. By Lemma \ref{Lem3.4}, there are $r_0,r_1,r_2\in\Z$ such that
$120n-15=r_0^2+6r_1^2+30r_2^2$.
 Clearly, $2\nmid r_0$ and $3\mid r_0$. Without loss of generality, we simply assume that
 $r_0\eq-1\pmod4$.
 Note that $r_1$ and $r_2$ are even since
$-2(r_1^2+r_2^2)\eq6r_1^2+30r_2^2\eq-15-r_0^2\eq0\pmod8$.
As $r_1^2\eq-r_0^2\eq(2r_0)^2\pmod5$, without loss of generality we may assume that $r_1\eq2r_0\pmod5$.
Write $r_1=2(s+5t)$ and $r_2=2x$ with $s,t,x\in\Z$, $0\ls s\ls 4$ and $s\eq r_0\pmod5$.

{\it Case} I. $r_0\eq-1\pmod 8$.

In this case, $r_0=120v+s^*$ with $v\in\Z$, where
$$s^*=\begin{cases}-24s+15&\mbox{if}\ s\in\{0,1,2\},\\-24s+135&\mbox{if}\ s\in\{3,4\}.\end{cases}$$
Thus
$$120n-15=(120v+s^*)^2+6(2(s+5t))^2+30(2x)^2$$
and hence
$$n=x^2+(a+t+6v)^2+(b+t-9v)^2+3(c-t-v)^2,$$
where
$$(a,b,c)=\begin{cases}(-s+1,2s-1,0)&\mbox{if}\ s\in\{0,1,2\},
\\(-s+7,2s-10,-1)&\mbox{if}\ s\in\{3,4\}.
\end{cases}$$
Set
$$y=a+t+6v,\ z=b+t-9v,\ w=c-t-v.$$ Then $n=x^2+y^2+z^2+3w^2$ and
$$2y+z+3w=2(a+t+6v)+(b+t-9v)+3(c-t-v)=2a+b+3c=1.$$

{\it Case} II. $r_0\eq3\pmod8$.

In this case, $r_0=120v+s_*$ with $v\in\Z$, where
$$s_*=\begin{cases}-45&\mbox{if}\ s=0,\\-24s+75&\mbox{if}\ s\in\{1,2,3,4\}.\end{cases}$$
Thus
$$120n-15=(120v+s_*)^2+6(2(s+5t))^2+30(2x)^2$$
and hence
$$n=x^2+(a+t+6v)^2+(b-2t+3v)^2+3(c-5v)^2,$$
where
$$(a,b,c)=\begin{cases}(-2,-1,2)&\mbox{if}\ s=0,
\\(-s+4,-s+2,s-3)&\mbox{if}\ s\in\{1,2,3,4\}.
\end{cases}$$
Set
$$y=a+t+6v,\ z=b-2t+3v,\ w=c-5v.$$ Then $n=x^2+y^2+z^2+3w^2$ and
$$2y+z+3w=2(a+t+6v)+b-2t+3v+3(c-5v)=2a+b+3c=1.$$

In view of the above, we have completed the proof of Theorem \ref{Th1.1}(ii). \qed

\begin{lemma}\label{Lem3.5} For any $n\in\Z^+$, we can write
$56n-24$ as $x^2+7y^2+56z^2$ with $x,y,z\in\Z$ and $2\nmid xy$.
\end{lemma}
\begin{proof} It is easy to see that $56n-24$ can be represented by the genus of
$f(x,y,z)=x^2+7y^2+14z^2$.
There are two classes in the genus of $f$, and the one not containing $f$ has an representative
$g(x,y,z)=2x^2+7y^2+7z^2$.
By Lemma \ref{Lem2.1}, $56n-24$ is represented by $f$ or $g$ over $\Z$.

In view of the identity
\begin{equation*}
f\l(\frac{2x-7y-7z}{4},\ \frac{2x+y+z}{4},\ \frac{-2y+2z}{4}\r)=g(x,y,z),
\end{equation*}
if $56n-24=g(x,y,z)$ for some $x,y,z\in\Z$ with $y\eq z\pmod2$ and $x\eq (y+z)/2\pmod2$ then $56n-24$ is
represented by $f$ over $\Z$.

Suppose that $56n-24=g(x,y,z)=2x^2+7y^2+7z^2$ with $x,y,z\in \Z$. Then $y^2+z^2\eq 2x^2\pmod8$ and hence
$y\eq z\pmod 2$.
Observe that
$$x^2\eq\f{y^2+z^2}2=\l(\f{y+z}2\r)^2+\l(\f{y-z}2\r)^2\pmod4.$$
If $(y+z)/2$ and $(y-z)/2$ are both odd, then $x^2\eq2\pmod4$ which is impossible.
Without loss of generality, we assume that $(y-z)/2$ is even. Hence $x\eq (y+z)/2\pmod2$ as desired.

By the above, $56n-24=f(u,v,w)=u^2+7v^2+14w^2$ for some $u,v,w\in\Z$. As $u^2+7v^2\eq
u^2-v^2\not\eq2\pmod4$, we have $w=2c$ for some $c\in\Z$.
Since $u^2+7v^2\eq 0\pmod8$ and $u^2+7v^2=56n-24-56c^2\not=0$, by \cite[Lemma 3.6]{S15} we can rewrite
$u^2+7v^2$ as $a^2+7b^2$ with $a$ and $b$ both odd.
Therefore $56n-24=a^2+7b^2+56c^2$ as desired.
\end{proof}

\noindent{\it Proof of Theorem} \ref{Th1.1}(iii). (1) By Lemma \ref{Lem2.2}, we can write $6n-1=r^2+s^2+6x^2$ with
$r,s,t\in \Z$. It is apparent that $3\nmid rs$ and $r\not\equiv s\pmod 2$.
 Without loss of generality, we may assume that $r\eq1\pmod 6$ and $s\eq 2\pmod 6$ (otherwise we may
 change signs of $r$ and $s$ to meet this purpose).
 Write $r=1-6w$ and $s=2-6v$ with $w,v\in\Z$. Then $6n-1=(1-6w)^2+(2-6v)^2+6x^2$ and hence
$$n=x^2+(1-2v-w)^2+2(v-w)^2+3w^2.$$
Let $y=1-2v-w$ and $z=v-w$. Then $n=x^2+y^2+2z^2+3w^2$ with $y+2z+3w=1$.

(2) By Lemma \ref{Lem3.5}, we can write $56n-24$ as $r_0^2+7r_1^2+56x^2$ with $r_0,r_1,x\in\Z$ and $2\nmid
r_0r_1$.
Since $r_0^2\eq -24\eq 2^2\pmod7$ and $2\nmid r_0r_1$, without loss of generality we may assume that
$r_0\eq2\pmod7$ and $r_1\eq-r_0\pmod4$.
Write $r_1=s+8t$ with $s\in\{\pm1,\pm3\}$ and $t\in\Z$, and
$r_0=r+56v$ with $v\in\Z$, where
$$r=\begin{cases}7s-12&\mbox{if}\ s\in\{-1,1\},\\-7s/3+16&\mbox{if}\ s\in\{-3,3\}.\end{cases}$$
Thus $56n-24=(56v+r)^2+7(s+8t)^2+56x^2$ and hence
$$n=x^2+(a+t-5v)^2+3(b+t+3v)^2+4(c-t+v)^2,$$
where
$$(a,b,c)=\begin{cases}(-(s-3)/2,\ (s-1)/2,\ 0)&\mbox{if}\ s\in\{-1,1\},
\\(s/3-1,\ 1,\ -(s-3)/6)&\mbox{if}\ s\in\{-3,3\}.
\end{cases}$$
Set
$$y=a+t-5v,\ z=b+t+3v,\ w=c-t+v.$$ Then $n=x^2+y^2+3z^2+4w^2$ and
$$y+z+2w=a+t-5v+(b+t+3v)+2(c-t+v)=a+b+2c=1.$$

In view of the above, we have completed the proof of Theorem \ref{Th1.1}(iii). \qed

\section{Proof of Theorem \ref{Th1.2}}

In light of the effective method in \cite[pp. 186--187]{Jones}, we can characterize
what natural numbers are represented by the form $5x^2+7y^2+70z^2$ locally. In particular, we have the following lemma.

\begin{lemma}\label{Lem4.1}
Set $f(x,y,z)=5x^2+7y^2+70z^2$.

{\rm (i)} If $n\in \bigcup_{k\in\N}\{4k+1, 4k+2,8k+4\}$, then $70n-2$ can be represented by $f$ locally.

{\rm (ii)} If $n\in \bigcup_{k\in\N}\{4k+3, 16k+8\}$, then $70n-32$ can be represented by $f$ locally.
\end{lemma}

\noindent{\it Proof of Theorem} \ref{Th1.2}(i). Let $f(x,y,z)=5x^2+7y^2+70z^2$.
By Lemma \ref{Lem2.4}, there is only one spinor genus in the genus of $f$.
It is easy to see that $2$ is the only anisotropic prime. Note also that
$$\{n\in\N:\ 16\nmid n\}=\bigcup_{k\in\N}\{4k+1,4k+2,4k+3,8k+4,16k+8\}.$$

{\it Case} 1. $n\in \bigcup_{k\in\N}\{4k+1, 4k+2, 8k+4\}$.

 In this case, $70n-2$ can be represented by $f$ locally (by Lemma \ref{Lem4.1}), and
$70n-2$ has bounded divisibility at $2$.
Thus, by Lemma \ref{Lem2.3}, if $n$ is large enough then $70n-2=5r_0^2+7r_1^2+70w^2$ for some
$r_0,r_1,w\in\Z$.
Since $r_0^2\eq1\pmod7$ and $r_1^2\eq-1\eq2^2\pmod5$, without loss of generality we may assume that
$r_0\eq1\pmod7$
and $r_1\eq2\pmod5$. It is clear that $r_0\eq r_1\pmod2$.

If $r_0\eq1\pmod{14}$, then $r_0=14u+1$ and $r_1=10v-3$ for some $u,v\in\Z$,
hence $70n-2=5(14u+1)^2+7(10v-3)^2+70w^2$ and thus $n=x^2+y^2+z^2+w^2$, where
$x=1+u-3v$, $y= 3u+v$ and $z=-2u$. Note that $x+3y+5z=1$.

If $r_0\eq8\pmod{14}$, then $r_0=14u+8$ and $r_1=10v+2$ for some $u,v\in\Z$,
hence $70n-2=5(14u+8)^2+7(10v+2)^2+70w^2$ and thus $n=x^2+y^2+z^2+w^2$, where
$x=u-3v$, $y= 2+3u+v$ and $z=-1-2u$. Obviously, $x+3y+5z=1$.

{\it Case} 2.  $n\in \bigcup_{k\in\N}\{4k+3,16k+8\}$.

In this case, $70n-32$ can be represented by $f$ locally (by Lemma \ref{Lem4.1}), and
 $70n-32$ has bounded divisibility at $2$.
Thus, by Lemma \ref{Lem2.3}, if $n$ is large enough then $70n-32=5r_0^2+7r_1^2+70w^2$ for some
$r_0,r_1,w\in\Z$.
Since $r_0^2\eq4^2\pmod7$ and $r_1^2\eq-1\eq3^2\pmod5$, without loss of generality we may assume that
$r_0\eq4\pmod7$
and $r_1\eq3\pmod5$. Obviously, $r_0\eq r_1\pmod2$.

If $r_0\eq4\pmod{14}$, then $r_0=14u+4$ and $r_1=10v-2$ for some $u,v\in\Z$,
hence $70n-32=5(14u+4)^2+7(10v-2)^2+70w^2$ and thus $n=x^2+y^2+z^2+w^2$, where
$x=1+u-3v$, $y= 1+3u+v$ and $z=-2u$. Note that $x+3y+5z=4$.

If $r_0\eq-3\pmod{14}$, then $r_0=14u-3$ and $r_1=10v+3$ for some $u,v\in\Z$,
hence $70n-32=5(14u-3)^2+7(10v+3)^2+70w^2$ and thus $n=x^2+y^2+z^2+w^2$, where
$x=-1+u-3v$, $y= 3u+v$ and $z=1-2u$. Clearly, $x+3y+5z=4$.

In view of the above, we have completed our proof of Theorem \ref{Th1.2}(i).
\medskip

\noindent{\it Proof of Theorem} \ref{Th1.2}(ii). (1) We first handle the case $(\mu,\nu)=(2,3)$.
In view of \cite[Theorem 6.3.1, p.\,148]{Ki}, there is only one spinor genus in the genus of $g(x,y,z)=2x^2+7y^2+84z^2$.
Note that $2$ is the only anisotropic prime for this genus.
For any $n\in\Z^+$, the odd number $42n-3$ can be represented by the genus of $g$. By Lemma \ref{Lem2.3},
if $n$ is large enough then $42n-3=2r_0^2+7r_1^2+84w^2$ for some $r_0,r_1,w\in\Z$. Clearly, $2\nmid r_1$.
As $r_0^2\eq -3/2\eq 3^2\pmod7$
and $r_1^2-r_0^2\eq0\pmod3$, without loss of generality, we may assume that $r_0\eq3\pmod7$ and
$r_1\eq-r_0\pmod3$.
Note that $r_0\eq (7r_1-15)/2\pmod{21}$. Write $r_1=1+2s+6t$ and $r_0=21v+(7s-15)/2$ with $s\in\{0,1,2\}$
and $t,v\in\Z$. Then
$$42n-3=2(21v+7s-4)^2+7(1+2s+6t)^2+84w^2$$
and hence
$$n=\l(1-s+t-4v\r)^2+\l(-s-2t-v\r)^2+\l(s+t+2v\r)^2+2w^2.$$
Set
$$x=1-s+t-4v,\ y=-s-2t-v,\ z=s+t+2v.$$
Then $n=x^2+y^2+z^2+2w^2$ with $x+2y+3z=1$.

(2) Now we handle the case $(\mu,\nu)=(2,5)$.
By \cite[Theorem 6.3.1, p.\,148]{Ki}, there is only one spinor genus in the genus of $5x^2+6y^2+60z^2$. It is easy to see
that the only anisotropic primes are $2,3,5$.
For any $n\in\Z^+$, the number $30n-1$ is represented by the genus of $5x^2+6y^2+60z^2$ and it is relatively prime to
$2\times3\times5$. Thus, by Lemma \ref{Lem2.3}, if $n$ is large enough then $30n-1=5r_0^2+6r_1^2+60w^2$
with $r_0,r_1,w\in\Z$. As $r_0$ or $-r_0$ is congruent to 1 modulo 6 and $r_1^2\eq -1\eq 2^2\pmod 5$,
without loss of generality we may write $r_0=6u+1$ and $r_1=5v+2$ with $u,v\in\Z$. Thus
$30n-1=5(6u+1)^2+6(5v+2)^2+60w^2$ and hence
$$n=(1+u+2v)^2+(2u-v)^2+(-u)^2+2w^2.$$
If we set $x=1+u+2v$, $y=2u-v$ and $z=-u$, then $n=x^2+y^2+z^2+2w^2$ with $x+2y+5z=1$.

(3) Now we consider the case $(\mu,\nu)=(3,4)$.
In light of \cite[Theorem 6.3.1, p.\,148]{Ki}, there is only one spinor genus in the genus of $x^2+26y^2+156z^2$.
Note that $2$ is the only anisotropic prime for this genus.
For any $n\in\Z^+$, the odd number $78n-3$ can be represented by the genus of $x^2+26y^2+156z^2$.
Thus, in view of Lemma \ref{Lem2.3}, provided that $n$ is sufficiently large we can write $78n-3$ as
$r_0^2+26r_1^2+156w^2$
with $r_0,r_1,w\in\Z$. As $r_0^2\eq-3\eq 7^2\pmod{13}$ and $r_0^2-r_1^2\eq0\pmod3$,
without loss of generality we may assume that $r_0\eq7\pmod{13}$ and $r_1\eq r_0\pmod3$.
Write $r_1=1+2s+3t$ with $s,t\in\Z$ and $|s|\ls1$. Then $78u+26s+7\eq2s+1\eq r_1\pmod3$
and $78u+26s+7\eq 7\pmod{13}$, and so $r_0=78u+26s+7$ for some $u\in\Z$. Thus
$$78n-3=(78u+26s+7)^2+26(1+2s+3t)^2+156w^2$$
and hence
$$n=(1+3s+t+7u)^2+(-s+t-5u)^2+(-t+2u)^2+2w^2.$$
If we put $x=1+3s+t+7u$, $y=-s+t-5u$ and $z=-t+2u$, then $n=x^2+y^2+z^2+2w^2$ and $x+3y+4z=1$.

(4) Now we consider the restriction $2y+z+w=1$. By Lemma \ref{Lem2.4}, there is only one spinor genus in
the genus
of $x^2+11y^2+55z^2$. It is easy to see that $11$ is the only anisotropic prime. For each $n\in\Z^+$,
the number $55n-10$ is represented by the genus of $x^2+11y^2+55z^2$ and it is relative prime to $11$. Hence, by Lemma
\ref{Lem2.3}, if $n$ is large enough then $55n-10=r_0^2+11r_1^2+55x^2$ with $r_0,r_1,x\in\Z$. Since
$r_0^2\equiv 1\pmod {11}$ and $r_1^2\equiv (2r_0)^2\pmod 5$, without loss of generality, we may assume
that
$r_0\equiv -1\pmod {11}$ and $r_1\equiv 2r_0\pmod 5$. Write $r_1=s+5t$ with $s,t\in\Z$ and $0\le s\le4$,
and $r_0=r+55v$ with $v\in\Z$, where $r=-22s+10$. As $55n-10=(r+55v)^2+11(s+5t)^2+55x^2$, we get
$$n=x^2+(a+t-2v)^2+(b-2t-v)^2+2(c+5v)^2,$$ where $(a,b,c)=(s,0,1-2s)$. Set
$$y=a+t-2v,\ z=b-2t-v,\ w=c+5v.$$ Then $n=x^2+y^2+z^2+2w^2$ and $2y+z+w=1$.

In view of the above, we have finished the proof of Theorem \ref{Th1.2}(ii). \qed

\begin{lemma}\label{Lem4.2} If $n\in\Z^+$ is large enough, then we can write $80n-15$ as
$x^2+16y^2+160z^2$ with $x,y,z\in\Z$ and $x\eq9\pmod{16}$.
\end{lemma}
\Proof. There are three classes in the genus of $x^2+y^2+32z^2$ with the following three corresponding
representatives:
\begin{align*}
f_1(x,y,z)&=x^2+y^2+32z^2,\\
f_2(x,y,z)&=2x^2+2y^2+9z^2+2yz-2zx,\\
f_3(x,y,z)&=x^2+4y^2+9z^2-4yz.
\end{align*}
Note that $f_1$ and $f_2$ constitute a spinor genus while another spinor genus in this genus contains $f_3$.
Using the methods from \cite{EHH}, one may easily deduce that $2\Z^2$ is the only spinor exceptional
square class of this genus and it is clearly that $2$ is the only anisotropic prime.
For any $n\in\Z^+$, the odd number $16n-3$ can be represented by the genus of $x^2+y^2+32z^2$.
Thus, by Lemma \ref{Lem2.3}, if $n$ is large enough then $16n-3=a^2+b^2+32c^2$ for some $a,b,c\in\Z$
with $2\nmid a$ and $2\mid b$. Clearly, $b/2$ is odd. Without loss of generality, we may assume that
$b/2\eq 2n-a\pmod4$. Note that
$$80n-15=5(16n-3)=5(a^2+b^2+32c^2)=(2a+b)^2+(a-2b)^2+160c^2.$$
Since $2a+b\eq4n\pmod8$, we have
$$(a-2b)^2\eq80n-15-(4n)^2\eq-15\eq 9^2\pmod{32}$$
and hence $a-2b\eq\pm9\pmod{16}$. Therefore
the desired result follows.  \qed

\medskip
\noindent {\it Proof of Theorem} \ref{Th1.2}(iii). (1) We first handle the case $P(x,y,z,w)=x+2y+w$.
By Lemma \ref{Lem4.2}, provided that $n\in\Z^+$ is large enough we can write $80n-15=r_0^2+16r_1^2+160z^2$
with $r_0,r_1,z\in\Z$ and $r_0\eq9\pmod{16}$. As $r_1^2\eq -r_0^2\eq (2r_0)^2\pmod5$,
without loss of generality we may assume that $r_1\eq2r_0\pmod{5}$
(otherwise we may adjust the sign of $r_1$ to meet our purpose).
Write $r_1=s+5t$ with $s\in\{-1,0,1,2,3\}$ and $t\in\Z$. Then $5+(-1)^s20+8s\eq 9\pmod{16}$
and $s\eq2(5+(-1)^s20+8s)\pmod{5}$. So $r_0=80u+5+(-1)^s20+8s$ for some $u\in\Z$.
Thus
$$80n-15=(80u+5+(-1)^s20+8s)^2+16(s+5t)^2+160z^2$$
and hence
$$n=(a+t-7u)^2+(b-t+2u)^2+2z^2+3(c+t+3u)^2,$$
where
$$(a,b,c)=\begin{cases}((3-s)/2,0,(s-1)/2)&\mbox{if}\ s\in\{-1,1,3\},
\\(-2-s/2,1,1+s/2)&\mbox{if}\ s\in\{0,2\}.
\end{cases}$$
If we set
$$x=a+t-7u,\ y=b-t+2u,\ w=c+t+3u,$$
then $n=x^2+y^2+2z^2+3w^2$ with
$$x+2y+w=a+t-7u+2(b-t+2u)+c+t+3u=a+2b+c=1.$$

(2) Now we handle the case $P(x,y,z,w)=y+z+w$.
By Lemma \ref{Lem2.4}, there is only one spinor genus in the genus of $x^2+11y^2+99z^2$. Note that $11$ is
the only anisotropic prime.
For any $n\in\Z^+$, the number $99n-54$ is not divisible by 11 and it can be represented by the genus of
$x^2+11y^2+99z^2$.
 Thus, by Lemma \ref{Lem2.3}, provided that $n$ is large enough we can write
 $99n-54$ as $r_0^2+11r_1^2+99x^2$ with $r_0,r_1,x\in\Z$.
As $r_0^2\eq1\pmod{11}$ and $r_1^2\eq r_0^2\eq (2r_0)^2\pmod3$, without loss of generality we may assume
that
$r_0\eq1\pmod{11}$ and $r_1\eq 2r_0\pmod3$. Write $r_1=s+9t$ with $s,t\in\Z$ and $0\ls s\ls 8$.
Also, we may write $r_0=99u-22s+s_0$ with $u\in\Z$, where $s_0=45$ if $s\ls4$, and $s_0=144$ otherwise.
Then
$$99n-54=(99u-22s+s_0)^2+11(s+9t)^2+99x^2$$
and hence
$$n=x^2+(a+2t+u)^2+2(b-t-5u)^2+3(c-t+4u)^2,$$
where
$$(a,b,c)=\begin{cases}(1,s-2,2-s)&\mbox{if}\ s\ls4,
\\(2,s-7,6-s)&\mbox{if}\ s>4.
\end{cases}$$
If we set $y=a+2t+u$, $z=b-t-5u$ and $w=c-t+4u$, then $n=x^2+y^2+2z^2+3w^2$
with $y+z+w=a+b+c=1$.

(3) Finally we handle the case $P(x,y,z,w)=y+2z+w$.
In view of Lemma \ref{Lem2.4}, there is only one spinor genus in the genus of $x^2+5y^2+30z^2$.
It is easy to see that $2$ is the only anisotropic prime.
For any $n\in\Z^+$, the odd number $30n-9$ can be represented by the genus of $x^2+5y^2+30z^2$.
In light of Lemma \ref{Lem2.3}, provided that $n$ is large enough we can write
$30n-9$ as $r_0^2+5r_1^2+30x^2$ with $r_0,r_1,x\in\Z$. As $r_0^2+5r_1^2$ is a positive multiple of $3$,
by \cite[Lemma 2.1]{S15} we can rewrite $r_0^2+5r_1^2$ as $(r_0')^2+5(r_1')^2$ with $r_0',r_1'\in\Z$
and $3\nmid r_0'r_1'$.
For simplicity, we just assume that $3\nmid r_0r_1$.
As $r_0^2\eq1\pmod5$ and $r_0^2-r_1^2\eq0\pmod 3$,
without loss of generality we may assume that $r_0\eq1\pmod5$ and $r_1\eq-r_0\pmod3$.
Write $r_1=s+6t$ with $s\in\{\pm1,\pm2\}$ and $t\in\Z$. Also, we may write $r_0=30u+s_0$ with $u\in\Z$,
where
$$s_0=\begin{cases}6-10s&\mbox{if}\ s\in\{-1,1\},
\\-9+5s&\mbox{if}\ s\in\{-2,2\}.
\end{cases}$$
Thus
$$30n-9=(30u+s_0)^2+5(s+6t)^2+30x^2$$
and hence
$$n=x^2+(a+2t+u)^2+2(b-t+u)^2+3(c-3u)^2,$$
where
$$(a,b,c)=\begin{cases}((3s-1)/2,(1-s)/2,(1-s)/2)&\mbox{if}\ s=\pm1,
\\(s/2,0,1-s/2)&\mbox{if}\ s=\pm2.
\end{cases}$$
If we set $y=a+2t+u$, $z=b-t+u$ and $w=c-3u$, then $n=x^2+y^2+2z^2+3w^2$ and
$y+2z+w=a+2b+c=1$.

The proof of Theorem \ref{Th1.2}(iii) is now complete. \qed

\subsection*{Acknowledgements}
The work is supported by the National Natural Science
Foundation of China (grant 11971222).
The authors would like to thank the referee for helpful comments.


\normalsize


\begin{thebibliography}{99}


\baselineskip=17pt


\bibitem{C} J. W. S. Cassels, {\it Rational Quadratic Forms}, Academic Press, London, 1978.
\bibitem{D27} L. E. Dickson, {\it Quaternary quadratic forms representing all integers}, Amer. J. Math.
    {\bf 49} (1927), 39--56.
\bibitem{D39} L. E. Dickson, {\it Modern Elementary Theory of Numbers}, Univ. Chicago Press, Chicago,
    1939.
\bibitem{Duke} W. Duke, {\it Hyperbolic distribution problems and half-integral weight Maass forms},
    Invent. Math. {\bf 92} (1988), 73--90.
\bibitem{DR} W. Duke and R. Schulze-Pillot, {\it Representations of integers by positive ternary quadratic
    forms and equidistribution of lattice points on ellipsoid}, Invent. Math. {\bf 99} (1990), 49--57.
\bibitem{EHH} A. G. Earnest, J. S. Hsia and D. C. Hung, {\it Primitive representations by spinor genera of
    ternary quadratic forms}, J. London Math. Soc. {\bf 50} (1994), 222--230.
\bibitem{Hanke} J. Hanke, {\it Some recent results about $(ternary)$ quadratic forms},
Number theory, CRM Proc. Lecture Notes 36, Amer. Math. Soc., Providence, RI, 2004, 147--164.

\bibitem{Jones} B. W. Jones, {\it The Arithmetic Theory of Quadratic Forms},  Carus
    Math. Monogr. 10, Math. Assoc. American, Buffalo, New York, 1950.
\bibitem{KS} B. Kane and Z.-W. Sun, {\it On almost universal mixed sums of squares and triangular
    numbers},
¡¡¡¡Trans. Amer. Math. Soc. {\bf 362} (2010), 6425--6455.
\bibitem{Ki} Y. Kitaoka, {\it Arithmetic of Quadratic Forms}, Cambridge Tracts in Math., 106,
Cambridge Univ. Press, Cambridge, 1993.
\bibitem{Ke}  M. Kneser, {\it Darstellungsmasse indefiniter quadratischer Formen}, Math. Z. {\bf 11}
    (1961), 188--194.
\bibitem{Oto} O. T. O'Meara, {\it Introduction to Quadratic Forms}, Springer, New York, 1963.
\bibitem{R} S. Ramanujan, {\it On the expression of a number in the form $ax^2+by^2+cz^2+dw^2$}, Proc.
    Cambridge Philos. Soc. {\bf 19} (1917), 11--21.
\bibitem{SP80} R. Schulze-Pillot,{\it Darstellung durch Spinorgeschlechter ternarer quadratischer Formen},
    J. Number Theory {\bf 12} (1980), 529--540.
\bibitem{SP00} R. Schulze-Pillot, {\it Exceptional integers for genera of integral ternary postive
    definite quadratic forms}, Duke Math. J. {\bf 102} (2000) 351--357.
\bibitem{Shimura} G. Shimura, {\it On modular forms of half-integral weight}, Ann. of Math. {\bf 97} (1973), 440--481.
\bibitem{Siegel} C. L. Siegel, {\it Uber die Klassenzahl algebraischer Zahlk\"orper}, Acta Arith. {\bf 1} (1935), 83--86.
\bibitem{SS} Y.-C. Sun and Z.-W. Sun, {\it Some variants of Lagrange's four squares theorem}, Acta Arith. {\bf 183} (2018), 339--356.
\bibitem{S15} Z.-W. Sun, {\it On universal sums of polygonal numbers}, Sci. China Math. {\bf 58} (2015),
    1367--1396.
\bibitem{S17a} Z.-W. Sun, {\it Refining Lagrange's four-square theorem}, J. Number Theory {\bf 175}
    (2017), 169--190.
\bibitem{S17b}Z.-W. Sun, {\it Restricted sums of four squares}, Int. J. Number Theory {\bf 15} (2019), 1863--1893.
\end{thebibliography}
\end{document}